\documentclass{amsart}

\usepackage{amsmath,amssymb,amsthm}

\newtheorem{prop}{Proposition}[section]
\newtheorem{thm}[prop]{Theorem}
\newtheorem{lem}[prop]{Lemma}
\newtheorem{cor}[prop]{Corollary}
\newtheorem{defi}[prop]{Definition}

\theoremstyle{definition}

\newtheorem{rem}[prop]{Remark}

\newcommand{\A}{\mathbb A}

\newcommand{\C}{\mathbb C}

\newcommand{\D}{\mathbb D}

\newcommand{\Hp}{\mathbb H}

\newcommand{\N}{\mathbb N}

\newcommand{\R}{\mathbb R}
\newcommand{\RR}{\mathcal R}

\renewcommand{\SS}{\mathcal S}

\newcommand{\lra}{\longrightarrow}
\newcommand{\ra}{\rightarrow}

\newcommand{\set}[1]{\{#1\}}
\newcommand{\bigset}[1]{\big\{#1\big\}}

\newcommand{\suchthatb}{\;\big|\;}

\newcommand{\define}{\;{\rm :=}\;}
\newcommand{\enifed}{\;{\rm =:}\;}

\newcommand{\abs}[1]{|#1|}

\newcommand{\Babs}[1]{\Big|#1\Big|}

\newcommand{\restrict}{\,|}

\newcommand{\im}{\Im m}

\DeclareMathOperator{\bd}{\mathrm{\!\!\restrict_{\partial\D}}}

\DeclareMathOperator{\crt}{\mathrm{Cr}}
\DeclareMathOperator{\crit}{\mathrm{Crit}}

\DeclareMathOperator{\inj}{\mathrm{Inj}}
\DeclareMathOperator{\inter}{\mathrm{inter}}
\DeclareMathOperator{\mix}{\mathrm{mix}}
\DeclareMathOperator{\virt}{\mathrm{virt}}

\DeclareMathOperator{\pr}{\mathrm{proj}}

\DeclareMathOperator{\id}{\mathrm{id}}

\DeclareMathOperator{\gl}{\mathrm{GL}}

\DeclareMathOperator{\sign}{\mathrm{sign}}

\DeclareMathOperator{\Hom}{\mathrm{Hom}}

\DeclareMathOperator{\Ed}{\mathrm{End}}

\DeclareMathOperator{\acc}{\mathrm{Acc}}

\DeclareMathOperator{\Image}{\mathrm{Im}}

\begin{document}

\author{Kai Zehmisch}
\address{Mathematisches Institut, Universit\"at zu K\"oln,
  Weyertal 86--90, 50931 K\"oln, Germany}
\email{kai.zehmisch@math.uni-koeln.de}

\title[Annulus property]{
  The annulus property of simple holomorphic discs
}

\date{September 19, 2011.}

\begin{abstract}
  We show 
  that any simple holomorphic disc admits the annulus property,
  i.e.\ each interior point is surrounded by an arbitrary small annulus
  consisting entirely of injective points.
  As an application we show
  that interior singularities of holomorphic discs can be resolved
  after slight perturbation of the almost complex structure.
  Moreover,
  for boundary points the analogue notion,
  the half-annulus property,
  is introduced
  and studied in detail.
\end{abstract}

\maketitle

%%%%%%%%%%%%%%%%%%%%%%%%%%%%%%%%%%%%%%%%%%%%%%%%%%%%%%%%%%%%%%%%%%%%%%

\section{Introduction\label{section:intro}}

In this article we consider holomorphic discs as Gromov \cite{gro85} introduced to symplectic geometry.
These are smooth (up to the boundary) maps $u$ defined on the closed unit disc $\D\subset\C$
which take values in an almost complex manifold $(M,J)$, map the boundary circle $\partial\D$
into a (maximally) totally real submanifold $L\subset M$, and solve the (non-linear)
Cauchy-Riemann equation
\[
\partial_su+J(u)\partial_tu=0,
\qquad
s+it\in\D.
\]
In order to make those discs applicable to symplectic geometry
not only certain compatibility conditions with the symplectic (or contact) structure are useful.
To ensure that the solution space of holomorphic discs has a meaningful structure
(e.g.\ the structure of a smooth finite-dimensional manifold)
the self-intersection behaviour
of a single solution should not be too exceptional.
If for example the solutions foliate the symplectic manifold
(as in case one talks about fillings with holomorphic discs, see Eliashberg \cite{elia90})
this has strong implications for the topology of the manifold.

As it is well-known (cf.\ the book of McDuff and Salamon \cite{mcdsal04})
the existence of a single {\bf injective point},
i.e.\ an immersed point which is not a double point,
is enough to show that the moduli space of solutions
is a smooth manifold,
provided the almost complex structure is perturbed suitably.
To find examples of {\bf somewhere injective} holomorphic discs
(i.e.\ discs which admit an injective point)
Lazzarini \cite{lazz00} found a method
to produce those out of a given holomorphic disc $u$
just by restricting the map $u$ to suitable subsets of $\D$.

Our interest lies in a closer understanding of the structure of 
{\bf simple} holomorphic discs,
discs with a dense subset of injective points.
For a given point we ask for a sufficiently small neighbourhood
that does not have too many double points.
This is the content of this work.
The central issue in doing so
(and not only for simple holomorphic discs
as e.g.\ studied by Lazzarini \cite{lazz00}, Kwon-Oh \cite{kwoh00}, and Oh \cite{oh97})
is that double points might accumulate not only in the interior,
but also at the boundary.
The first case is obstructed by the existence
of a dense set of injective points, i.e.\ by simplicity.
This is due to Micallef and White \cite{micwh95}.
But on the boundary the situation is completely different
as we shall explain before we come to our main results.

By the work of Micallef-White \cite{micwh95}
the intersection behaviour in the interior is understood
and a unique continuation principle is valid (well-known in complex analysis).
For boundary intersections this holds only under extra conditions:
Consider the local situation of two holomorphic maps
$z\mapsto z$, $z\mapsto -z$ both defined on the closed upper half plane $\Hp$.
We see that the images only have the real line in common
opposite to plenty of intersection points.
This means that the unique continuation principle is violated.
Globally this is also the case as the map $(\D,\partial\D)\ra (\C P^1,\R P^1)$
induced by $\Hp\ni z\mapsto z^2\in\C$ shows.
But if one can exclude this self-matching phenomenon
a relative version of the similarity principle
due to Carleman \cite{car39} can be applied,
and hence contradicts simplicity of the holomorphic disc.
We remark
that this relative Carleman similarity principle
for certain technical reasons only works
provided the boundary accumulation points are immersed.

In the situation where mixed intersections appear,
which are those where the boundary circle intersects the disc at the interior,
no analog of a unique continuation principle is known.
Moreover, a formulation of such a theorem for an arbitrary holomorphic disc map $u$
should involve the structure of the set $u^{-1}\big(u(\partial\D)\big)$,
the so-called {\bf net},
as the example $\Hp\ni z\mapsto z^3\in\C$ shows.
The regions between the $e^{ik\pi/3}\R^+$-axes, $k=0,1,2,3$, are overlapping and self-matching, respectively.
But in general the net has a very rich structure
which is still not fully understood,
c.f.\ \cite{kwoh00}.
But if one is only interested in subquestions
concerning the local behaviour of holomorphic discs
one can avoid those problems as done e.g.\ in \cite{lazz00} by Lazzarini,
where he showed how to reduce a given holomorphic disc to obtain a simple one.

Finally we point out
that there a geometric situations
in which the holomorphic discs behave particularly nice.
One class of examples are almost complex manifolds
which allow a Schwarz reflection principle.
This is the case for a real analytic boundary condition $L$
and an almost complex structure integrable near $L$, see \cite{pan94},
or if one can find an anti-holomorphic involution with a fix-point set containing $L$, c.f.\ \cite{frsch05}.
After an extension by reflection
all self-intersection points lie in the interior
and the results of Micallef and White \cite{micwh95}
can be applied.
Another situation where no mixed intersections can appear are holomorphic discs inside a
strictly pseudo-convex domain $\Omega$ such that $L\subset\partial\Omega$,
c.f.\ Eliashberg \cite{elia90}.
But in this work no such restrictions to the almost complex structure is made
and the aim is to prove results in full generality.
In other words, we have to deal with all kinds of double points.

Our first result concerns the structure of injective points near a given interior point.
We say a holomorphic disc has the {\bf annulus property}
if around any interior point one can find
an arbitrary small annulus consisting entirely of injective points.
For a precise definition we refer to Definition \ref{def:annulusprop}.

\begin{thm}
  \label{annulusintro}
  A holomorphic disc has the annulus property
  if and only if the disc is simple.
\end{thm}

The annulus property allows one to perform
local constructions with a given holomorphic disc.
Of primary interest (in particular in four-dimensional topology)
is the existence of immersed holomorphic discs.

\begin{cor}
  \label{mcdnonlin}
  For any simple $J$-holomorphic disc $u:(\D,\partial\D)\ra(M,L)$
  there exist a smooth almost complex structure $\tilde{J}$
  (still making $L$ totally real)
  and a simple $\tilde{J}$-holomorphic disc $\tilde{u}:(\D,\partial\D)\ra(M,L)$
  such that
  \begin{itemize}
  \item [(i)]
    $\tilde{u}$ has no interior critical point,
    i.e.\ $\tilde{u}\restrict_{B_1(0)}$ is an immersion,
  \item [(ii)]
    $\tilde{u}$ is arbitrary $C^2$-close to $u$, and
  \item [(iii)]
    $\tilde{J}$ is arbitrary $C^1$-close to $J$.
  \end{itemize}
\end{cor}

To obtain the result one just removes all the finitely many interior critical points of $u$.
By a theorem of McDuff (see \cite[Theorem 4.1.1]{mcd94}),
about small discs
near each interior critical point
the $J$-holomorphic disc $u$ can be locally approximated
by $J$-holomorphic immersions.
Now using small annuli around
the interior critical points as in Theorem \ref{annulusintro}
the cut and paste procedure from \cite[Lemma 4.3]{mcd91} and \cite[Corollary 4.2.1]{mcd94}
yields the claim.
As the construction shows,
$\tilde{u}$ coincides with $u$
and $\tilde{J}$ coincides with $J$
away from the neighbourhoods,
where the local perturbations take place.
The annulus property thereby is used in a crucial way
in order to extend the almost complex structure over the annuli to obtain a smooth $\tilde{J}$.

For a similar formulation of the annulus property along the boundary
one has to exclude the case
that the holomorphic disc-map takes the same values along two disjoint boundary segments
(recall the example $z\mapsto z^2$ from above).
We will say the holomorphic disc is {\bf not self-matching},
see Definition \ref{def:selfm} for a precise formulation.
We say a holomorphic disc has the {\bf half-annulus property}
if around any boundary point one can find
an arbitrary small half-annulus consisting entirely of injective points,
see  Definition \ref{def:halfannulusprop}.

\begin{thm}
  \label{halfannulusintro}
  A holomorphic disc has the half-annulus property
  if and only if the disc is simple and not self-matching.
\end{thm}

We remark that both theorems above are crucial in developing a higher-order-intersection theory
for holomorphic discs.
This is expected to generalise the filling-with-holomorphic-discs method considerably, see \cite{zehm08}.
For a generic almost complex structure the moduli space of somewhere injective holomorphic discs
(representing a certain relative homotopy class)
with e.g.\ a boundary singularity, a tangent self-intersection, or
a more general constraint on their jet-prolongations is a smooth manifolds
of the expected dimension.
By an argument due to McDuff-Salamon (see \cite[Lemma 3.3.4]{mcdsal04})
the (half-)annulus property is used to perform a linearised version of the argument of Corollary \ref{mcdnonlin}
to get the manifold structure.
We point out that the half-annulus property is especially made for this argument.
Moreover, {\it a posteriori}
all somewhere injective holomorphic discs are simple
and even simple along the boundary for a generic almost complex structure, see \cite{zehm08}.

This article is organised as follows:
A precise definition of simple holomorphic discs is given in Section \ref{section:characterisation},
where we show the equivalence to the name-giving property of not being somewhere locally multiply covered.
The annulus property is introduced in Section \ref{annuluspropertysubsec}
and the first steps of the proof of Theorem \ref{annulusintro} (which is identical with Theorem \ref{annulusproperty})
are made.
Section \ref{sec:boundint} treats the intersection behaviour at boundary intersection-points
based on the relative Carleman similarity principle.
Section \ref{section:mixed} does the analogue for mixed intersections
and the proof of Theorem \ref{annulusintro}
with a set-theoretical argument based on Section \ref{sec:boundint} is finished.
In Section \ref{halfannuluspropertysubsec} the half-annulus property is introduced and,
as a first step for Theorem \ref{halfannulusintro},
the equivalence to simplicity along the boundary is shown.
Sections \ref{injpointsalongtheboundary} and \ref{simplicityalongtheboundary}
are devoted to the concept of injective points of the boundary map of holomorphic discs.
Section \ref{notselfmatching} discusses the relation between simplicity along the boundary
and locally multiply covered boundary maps and
completes the proof of Theorem \ref{halfannulusintro}.

%%%%%%%%%%%%%%%%%%%%%%%%%%%%%%%%%%%%%%%%%%%%%%%%%%%%%%%%%%%%%%%%%%%%%%

\section{Characterisation of simplicity\label{section:characterisation}}

Due to the fundamental result of Micallef-White \cite{micwh95}
the local behaviour of simple holomorphic spheres is well-understood,
and studied in full detail in the book of McDuff and Salamon \cite{mcdsal04}.
Some of these local properties remain valid for interior points of holomorphic discs.
The aim of this section is to give an overview of what is known about this.

Let us consider an almost complex manifold $(M,J)$
which contains a totally real submanifold $L$.
We would like to study the local behaviour of a $J$-holomorphic disc $u:(\D,\partial\D)\ra(M,L)$.
For this let us recall some definitions:
A {\bf critical point of} $u$ is a point $z\in\D$ such that $du(z)=0$.
Synonymously, we say {\bf singular point} or just {\bf singularity}.
We denote the set of critical points by
                                                            \begin{gather*}
                                                              \crit(u)\define
                                                              \bigset{
                                                                z\in\D
                                                                \suchthatb
                                                                du(z)=0
                                                              }
                                                            \end{gather*}
and the preimage of critical values by
                                                            \begin{gather*}
                                                              \crt(u)\define u^{-1}\Big(u\big(\crit(u)\big)\Big).
                                                            \end{gather*}
Notice, any point $z\in\D$
which is not a critical point of $u$,
is automatically immersed by $u$.
This is due to the fact that $u$ is $J$-holomorphic,
i.e.\ $u$ solves
                                                            \begin{gather*}
                                                              \partial_tu(z)=J\big(u(z)\big)\partial_su(z)
                                                            \end{gather*}
for all $z\in\D$.
Therefore,
by the inverse function theorem,
$u$ is a local embedding near its non-critical points.

Holomorphic maps satisfy the unique continuation principle well-known for analytic functions in one variable.
That the principle is valid for $J$-holomorphic maps is due to Carleman \cite{car39}.
For holomorphic discs we obtain therefore, see Lazzarini \cite[Theorem 3.5]{lazz00}:

\begin{prop}
  \label{pointsfinite}
  For a non-constant $J$-holomorphic disc $u:(\D,\partial\D)\ra(M,L)$ the following sets are finite:
  \begin{itemize}
  \item [(i)]
    \label{critpointsfinite}
    the set of critical points $\crit(u)$,
  \item [(ii)]
    \label{prepointsfinite}
    the preimage $u^{-1}(p)$ for all $p\in M$,
  \item [(iii)]
    \label{precritpointsfinite}
    and, in particular, the preimage $\crt(u)$ of critical values.
  \end{itemize}
\end{prop}

As it is well-known (see \cite{mcdsal04,zehm08})
the existence of an {\bf injective point},
a point $z\in\D$ such that
                                                            \begin{gather}
                                                              \label{sipoint}
                                                              u^{-1}\big(u(z)\big)=\set{z}
                                                              \quad\text{and}\quad
                                                              du(z)\neq0,
                                                            \end{gather}
can be used
to provide the solution space of the Cauchy-Riemann operator
(i.e.\ the moduli space of holomorphic discs)
with the structure of a finite dimensional manifold (at least locally).

\begin{prop}
  \label{injuopen}
  The set of all injective points of $u$
                                                            \begin{gather*}
                                                              \inj(u)\define
                                                              \bigset{
                                                                z\in\D
                                                                \suchthatb
                                                                z\;\;
                                                                \text{\it is an injective point of}
                                                                \;\;u
                                                              }
                                                            \end{gather*}
 is open in $\D$.
\end{prop}

\begin{proof}
  Let $z_0$ be an injective point.
  We claim that there exists $r>0$ such that $u^{-1}\big(u(z)\big)=\set{z}$ for all $z\in B_r(z_0)\cap\D$.
  Otherwise, we could find sequences $z_{\nu}\ra z_0$ and $w_{\nu}\ra w_0$ in $\D$
  such that $z_{\nu}\neq w_{\nu}$ as well as $u(z_{\nu})=u(w_{\nu})$ for all $\nu\in\N$.
  But by the first condition in \eqref{sipoint} this would imply that $w_0=z_0$.
  This is a contradiction, because by the second condition in \eqref{sipoint}
  the restriction $u\restrict_V$ is an embedding (and hence injective)
  for some open neighbourhood $V$ of $z_0$ in $\D$.
\end{proof}

A $J$-holomorphic disc $u$ is called {\bf somewhere injective} if $u$ has an injective point,
i.e.\ if $\,\inj(u)\neq\emptyset$.
In order to find conditions under which a holomorphic disc is somewhere injective
we will study
the set of {\bf self-intersection points of} $u$
                                                            \begin{gather*}
                                                              S(u)\define
                                                              \bigset{
                                                                (z_1,z_2)\in\D\times\D\setminus\Delta_{\D}
                                                                \suchthatb
                                                                u(z_1)=u(z_2)
                                                              }
                                                            \end{gather*}
and
the projection to the first coordinate
                                                            \begin{gather*}
                                                              S_1(u)\define\pr_1S(u).
                                                            \end{gather*}
By the mean value theorem applied to the coordinate functions of $u$ we see that the closure of $S(u)$ is contained in
the disjoint union
                                                            \begin{gather}
                                                              \label{clousureofsu}
                                                              \overline{S(u)}\subset S(u)\sqcup\Delta_{\crit(u)}.
                                                            \end{gather}
The inclusion in \eqref{clousureofsu} is in general strict as the following injective map shows:
$\D\ni z\mapsto (z^2,z^3)$.
Moreover,
we observe that the complement of the set of injective points has the following description:
                                                            \begin{gather}
                                                              \label{dwithoutinju}
                                                              \D\setminus\inj(u)=
                                                              \crit(u)\cup S_1(u)\supset
                                                              \overline{S_1(u)}.
                                                            \end{gather}
Therefore, we see again that $\inj(u)$ is open.
Further, we remark that
                                                            \begin{gather*}
                                                              \crt(u)=
                                                              \crit(u)\cap S_1(u),
                                                            \end{gather*}
a fact that we will use throughout our arguments.

Following Lazzarini \cite{lazz00} we define now:

\begin{defi}
  \label{verydefofsimpledisc}
  A $J$-holomorphic disc $u$ is {\bf simple} if $\,\inj(u)$ is dense in $\D$.
\end{defi}

The first non-trivial observation concerning simplicity is due to Micallef and White \cite{micwh95}:

\begin{thm}
  \label{mwmagerthm}
  For any simple $J$-holomorphic disc $u$ the set
                                                            \begin{gather*}
                                                              S_{\inter}(u)\define
                                                              S(u)\cap
                                                              \big(B_1(0)\times B_1(0)\big)
                                                            \end{gather*}
  of {\bf interior self-intersection points of} $u$ is discrete in $B_1(0)\times B_1(0)$.
\end{thm}

Equivalently we could say that for all $r\in (0,1)$ the set $S(u)\cap(\D_r\times\D_r)$ is finite.
This means for a simple $J$-holomorphic disc $u$ there exist no sequences $z_{\nu}\ra z_0$ and $w_{\nu}\ra w_0$
in $B_1(0)$ with $z_{\nu}\neq w_{\nu}$ for all $\nu\in\N$ and $z_0,w_0\in B_1(0)$
such that $u(z_{\nu})=u(w_{\nu})$,
regardless whether $z_0$ and $w_0$ coincide or not.
A proof of Theorem \ref{mwmagerthm}
avoiding the theory of minimising area-like functionals is given by Lazzarini in \cite[Theorem E.1.2]{mcdsal04}.

\begin{rem}
  \label{pointedembedding}
  A particular consequence of Theorem \ref{mwmagerthm} is that for a simple disc $u$
  any point in $\crit(u)\cap B_1(0)$
  has a pointed neighbourhood $U^*$ such that the restriction $u\restrict_{U^*}$ is an embedding,
  i.e.\ is injective.
\end{rem}

Sometimes, a weaker version of Theorem \ref{mwmagerthm} is sufficient
in order to get informations about the geometric structure of a holomorphic disc, see \cite[Lemma 2.4.3]{mcdsal04}.

\begin{lem}
  \label{weakversion}
  Let $u$ be a $J$-holomorphic disc.
  Assume that there are sequences $z_{\nu}\ra z_0$ and $w_{\nu}\ra w_0$
  in $B_1(0)$ with $z_{\nu}\neq w_{\nu}$ for all $\nu\in\N$ and $z_0,w_0\in B_1(0)$
  such that $u(z_{\nu})=u(w_{\nu})$.
  Assume in addition that $du(z_0)\neq 0$.
  Then there exists an analytic function $\varphi$ defined on an open neighbourhood $V$ of $w_0$
  such that $\varphi(w_0)=z_0$ and $u=u\circ\varphi$ on $V$.
\end{lem}

The second proof of this fact given in \cite{mcdsal04} relies on the Carleman similarity principle.
In section \ref{sec:boundint} we extend this approach to cases,
where the boundary is present.

Lemma \ref{weakversion} can be used to give an equivalent characterisation of simplicity
which is more practical in the applications:

\begin{prop}
  \label{equivcharofsimple}
  A $J$-holomorphic disc $u$ is simple if and only if there are no two nonempty
  disjoint open subsets $U_1$ and $U_2$ of $\D$
  such that $u(U_1)=u(U_2)$.
\end{prop}

\begin{proof}
  The $J$-holomorphic disc $u$ is not simple if and only if $\inj(u)$ is not dense in $\D$,
  i.e.\ the complement $\D\setminus\inj(u)$ has an interior point.
  Because of the fact that $\crit(u)$ is finite, $u$ is not simple if and only if $S_1(u)$ has an interior point,
  see \eqref{dwithoutinju}.
  
  If there are two nonempty disjoint open subsets $U_1$ and $U_2$ of $\D$
  such that $u(U_1)=u(U_2)$ we get in particular $U_1\subset S_1(u)$
  which means that $S_1(u)$ has an interior point.
  By the previous discussion this implies that $u$ is not simple.
  
  On the other hand, if $S_1(u)$ has an interior point then in particular $S_{\inter}(u)$ has an accumulation point.
  Because $\crit(u)$ is finite we can assume that we are in the situation of Lemma \ref{weakversion}.
  Now $du(z_0)\neq 0$ implies that $z_0\neq w_0$ and we can further assume that $U\define\varphi(V)$
  and $V$ are disjoint.
  Notice, that any analytic function is open.
  Consequently, we found nonempty disjoint open sets $U$ and $V$ such that $u(V)=u(U)$.
\end{proof}

We see that our definition of simplicity coincides with the one given by Lazzarini in \cite{mcdsal04}.
Combining Theorem \ref{mwmagerthm} with Proposition \ref{equivcharofsimple}
we get a further characterisation of simplicity:

\begin{cor}
  A $J$-holomorphic disc $u$ is simple if and only if $S_{\inter}(u)$ is a discrete subset of $B_1(0)\times B_1(0)$.
\end{cor}

%In Section \ref{annuluspropertysubsec} we will develop a further characterisation of simplicity.

%%%%%%%%%%%%%%%%%%%%%%%%%%%%%%%%%%%%%%%%%%%%%%%%%%%%%%%%%%%%%%%%%%%%%%
%%%%%%%%%%%%%%%%%%%%%%%%%%%%%%%%%%%%%%%%%%%%%%%%%%%%%%%%%%%%%%%%%%%%%%

\section{The annulus property\label{annuluspropertysubsec}}

Throughout this section we consider a $J$-holomorphic disc $u:(\D,\partial\D)\ra(M,L)$,
where $(M,J)$ is an almost complex manifold containing a totally real submanifold $L$.
The aim is to show that for any interior point of $\D$ 
there exist arbitrary small surrounding annuli
which consist completely of injective points of $u$,
provided $u$ is simple.
In other words, we will prove that $u$ has the annulus property.
This is the analogous statement to \cite[Corollary 2.5.4]{mcdsal04},
where that case of holomorphic spheres is considered.

\begin{defi}
  \label{def:annulusprop}
  We will say that $u$ has the {\bf annulus property}
  if for any point $z_0\in B_1(0)$ and any given $\varepsilon>0$
  with $\overline{B_{\varepsilon}(z_0)}\subset B_1(0)$
  there exists a closed subset $A_{\varepsilon,z_0}$ diffeomorphic to the annulus $S^1\times [0,1]$
  such that the following conditions are satisfied:
  \begin{itemize}
  \item [(i)]
    $A_{\varepsilon,z_0}\subset B^*_{\varepsilon}(z_0)$,
    where $B^*_{\varepsilon}(z_0)\define  B_{\varepsilon}(z_0)\setminus\set{z_0}$,
  \item [(ii)]
    a boundary component of $A_{\varepsilon,z_0}$ has winding number $1$ around $z_0$,
  \item [(iii)]
    $u\restrict_{A_{\varepsilon,z_0}}$ is an embedding, and
  \item [(iv)]
    $u^{-1}\big(u(A_{\varepsilon,z_0})\big)=A_{\varepsilon,z_0}$.
  \end{itemize}
\end{defi}

We see that any $J$-holomorphic disc $u$
which has the annulus property is simple.
This is because for all $z_0\in B_1(0)$ and for all $\varepsilon>0$
the annulus $A_{\varepsilon,z_0}$ is contained in $\inj(u)$
so that the set of injective points is dense,
see Definition \ref{verydefofsimpledisc}.
The converse is also true:

\begin{thm}
  \label{annulusproperty}
  A $J$-holomorphic disc is simple if and only if it has the annulus property.
\end{thm}

\begin{proof}[{\bf Proof of Theorem \ref{annulusproperty} (part I)}]
  The above discussion shows that it is enough to show that
  any simple $J$-holomorphic disc $u$ has the annulus property.
  We start our proof with some preliminaries.
  Let $z_0\in B_1(0)$ be any point and take $\varepsilon>0$ such that $\overline{B_{\varepsilon}(z_0)}\subset B_1(0)$.
  By Proposition \ref{pointsfinite} %\eqref{precritpointsfinite}
  we can choose $\varepsilon>0$ so small such that
                                                            \begin{gather}
                                                              \label{firstconclusion}
                                                              B^*_{\varepsilon}(z_0)\cap\crt(u)=\emptyset.
                                                            \end{gather}
  In particular, $u\restrict_{B^*_{\varepsilon}(z_0)}$ is an immersion but potentially with self-intersections.
  By Theorem \ref{mwmagerthm} we have that $S_{\inter}(u)$ is discrete and, consequently,
  \[
  S(u)\cap\big(B_{\varepsilon}(z_0)\times B_{\varepsilon}(z_0)\big)
  \]
  is a finite set.
  Hence, we can assume after shrinking $\varepsilon>0$ if necessary
  that the restricted map $u\restrict_{B^*_{\varepsilon}(z_0)}$ is injective.
  Or equivalently, $u\restrict_{B^*_{\varepsilon}(z_0)}$ is an embedding, cf.\ Remark \ref{pointedembedding}.
  
  In order to prove the annulus property
  it remains to study the subset $u^{-1}\big(u(\A)\big)$ of $\D$,
  where
                                                            \begin{gather*}
                                                              \A\define
                                                              \overline{B_{\varepsilon/2}(z_0)}\setminus B_{\varepsilon/4}(z_0)
                                                              \subset
                                                              B^*_{\varepsilon}(z_0).
                                                            \end{gather*}
  The aim is to show that there is a subannulus of $\A$
  which has winding number $1$ and consists entirely of injective points.
  Because $\inj(u)$ is open (see Proposition \ref{injuopen})
  for this it will be enough to show that there is an embedded circle in $\A$
  which has winding number equal to $1$,
  such that any point on the circle is injective.
  
  In order to provide the space of self-intersections with a useful structure
  we will introduce some terminology:
  On the set $\A\times\D$ we define the following correspondence:
                                                            \begin{gather*}
                                                              \SS\define
                                                              S(u)\cap\A\times\D.
                                                            \end{gather*}
  By simplicity of $u$ this set has no interior point in $\A\times\D$,
  see Proposition \ref{equivcharofsimple}.
  Furthermore, in view of \eqref{firstconclusion} the set $\SS$ must be closed.
  Therefore, the set of accumulation points
                                                            \begin{gather*}
                                                              \RR\define\acc(\SS)
                                                            \end{gather*}
  of $\SS$ is contained in $\SS$.
  Theorem \ref{mwmagerthm} implies now that $\acc(\SS)$ is a subset of $\A\times\partial\D$
  and the intersection $\SS\cap\A\times B_1(0)$ is discrete.
  Hence, we get
                                                            \begin{gather}
                                                              \label{disjointunionofintpoints}
                                                              \RR\subset\big(\SS\cap\A\times\partial\D\big)
                                                              \quad\text{and}\quad
                                                              \RR\cap\A\times B_1(0)=\emptyset.
                                                            \end{gather}
  Consequently, the projection to the first component
                                                            \begin{gather*}
                                                              R_1\define\pr_1\!\RR
                                                            \end{gather*}
  which is a compact subset of $\A$,
  has no interior point (viewed as a subspace of $\A$).
  This means that the set $\A\setminus R_1$ is dense in $\A$,
  or equivalently the set $R_1$ is nowhere dense in $\A$.
  Roughly speaking, $R_1$ contains no $2$-dimensional components.
  In fact, $1$-dimensional pieces are not there as well:

  \begin{lem}
    \label{noembeddingintor1}
    There exists no embedding $c:[-1,1]\ra\A$ with $c([-1,1])\subset R_1$.
  \end{lem}

  Taking the lemma to be granted we see together with \eqref{disjointunionofintpoints}
  that the compact set
                                                            \begin{gather*}
                                                              S_1\define\pr_1\!\SS
                                                            \end{gather*}
  is nowhere dense in $\A$ and there exists no embedding $c:[-1,1]\ra\A$
  such that its image $c([-1,1])$ is contained in $S_1$.

  Before we continue with the construction of the desired annulus $A_{\varepsilon,z_0}$
  we will give the proof of Lemma \ref{noembeddingintor1}.
  It relies an the relative Carleman similarity principle and is postponed to section \ref{sec:boundint}.
\end{proof}

\begin{rem}
  The set $\SS\cap\A\times\partial\D$ has no interior point in $\A\times\partial\D$.
  Otherwise, we would find open subsets $U\subset\A$ and $I\subset\partial\D$
  such that $u(U)=u(I)$, and hence, that $u(U)\subset L$.
  This contradicts the fact that $L$ is totally real.
  Further, the set
                                                            \begin{gather*}
                                                              \pr_1\!\big(S(u)\cap\A\times\partial\D\big)=
                                                              \bigset{
                                                                x\in\A
                                                              \suchthatb
                                                              \exists
                                                              y\in\partial\D:
                                                              u(x)=u(y)
                                                              }
                                                            \end{gather*}
  has no interior point in $\A$.
  Otherwise, there would exist an open subset $U$ of $\A$ such that $u(U)\subset u(\partial\D)\subset L$,
  which is again a contradiction.
\end{rem}

%%%%%%%%%%%%%%%%%%%%%%%%%%%%%%%%%%%%%%%%%%%%%%%%%%%%%%%%%%%%%%%%%%%%%%
%%%%%%%%%%%%%%%%%%%%%%%%%%%%%%%%%%%%%%%%%%%%%%%%%%%%%%%%%%%%%%%%%%%%%%

\section{Boundary intersections \label{sec:boundint}}

Let $(M,J)$ be an almost complex manifold of real dimension $2n$ containing a totally real submanifold $L$.
In this section we consider two $J$-holomorphic embeddings
\[
u,v:\big(\D^+,[-1,1],0\big)\lra(M,L,p),
\]
of the closed upper half-disc $\D^+=\D\cap\set{\im(z)\geq0}$.
We are interested in the accumulation behaviour of the intersection points of $u$ and $v$ near their boundaries.
So we consider two sequences $z_{\nu}\ra 0$ and $w_{\nu}\ra 0$ of points in $\D^+\setminus\set{0}$
such that $u(z_{\nu})=v(w_{\nu})$ for all $\nu\in\N$.

If we stick to the $4$-dimensional case for the moment
we see that $u$ and $v$ must be tangent at $0$,
i.e.\ the images of $du(0)$ and $dv(0)$ must coincide.
Indeed,
assuming the contrary
the $J(p)$-invariance of the tangent spaces of the half-discs at $p$
would then imply that the intersection of $du(0)\set{\C}$ with $dv(0)\set{\C}$ is $\set{0}$.
But by the dimension assumption
this means that $u$ and $v$ would be transverse at $(0,0)$.
This is a contradiction.

\begin{rem}
  In the above argument the presence of the boundary was not used.
  Taking the boundary into account we see
  that the linear maps $du(0)$ and $dv(0)$ are collinear over $\R$.
\end{rem}

This is not just a $4$-dimensional phenomenon as the following lemma will show.
So independent of the dimension the holomorphic half-discs are tangent at the accumulation point.

\begin{lem}
  \label{signofdelta}
  Let $u,v$ be as above.
  Then there exists a non-zero real number $\delta$ such that $du(0)=\delta dv(0)$.
\end{lem}

\begin{proof}
  We will show that $du(0)$ and $dv(0)$ are collinear over $\C$,
  i.e.\ there exists a complex number $a\neq0$ such that $du(0)= dv(0)\cdot a$.
  Because the partial derivatives $u_s(0)$ and $v_s(0)$ are contained in $T_pL$
  the number $a$ must be real.
  
  First of all we can find a local chart about $p$
  in which the totally real boundary $L$ corresponds to $\R^n\subset\C^n$,
  and in which the almost complex structure $J$ is the multiplication by $i$
  at least for all real points,
  see \cite[Exercise B.4.10]{mcdsal04} or \cite[Proposition 3.3]{lazz00}.
  Therefore, it is enough to consider $J$-holomorphic maps
                                                            \begin{gather*}
                                                              u,v:
                                                              (\D^+,[-1,1],0)\lra(\C^n,\R^n,0),\\
                                                              u=(u_1,\ldots,u_n),
                                                              v=(v_1,\ldots,v_n),
                                                            \end{gather*}
  where $J$ is an almost complex structure on $\C^n$ such that $J=i$ on $\R^n$.

  Arguing by contradiction we suppose that $du(0)\set{\C}$ and $dv(0)\set{\C}$
  do not coincide.
  So they span a real $4$-dimensional subspace of $\C^n$,
  because the case of a real $3$-dimensional subspace is excluded by $i$-invariance.
  Now we wish to apply the relative Carleman similarity principle (see \cite[Theorem A.2]{abb04})
  to each of the solutions 
  $u$ and $v$ of
  \begin{gather*}
    w_s+J(w)w_t=0,\\
    w([-1,1])\subset\R^n,\\
    w(0)=0,
  \end{gather*}
  separately.
  Using a smooth endomorphism field which conjugates $J$ to $i$ and is the identity along $\R^n$
  we can assume that $u$ and $v$
  solve the linear equation
  \begin{gather*}
    w_s+iw_t+Aw=0,\\
    w([-1,1])\subset\R^n,\\
    w(0)=0,
  \end{gather*}
  where $A$ is a smooth real matrix valued function on $\D^+$ (depending on $u$, resp. $v$),
  see \cite[p.~24]{mcdsal04}.
  Now by the relative Carleman similarity principle we find
  \begin{itemize}
  \item
    a continuous complex matrix valued function $B$ of invertible matrices
    with $B(0)=\id$,
  \item
    and an analytic $\C^n$-valued function $f$ with $f(0)=0$,
  \end{itemize}
  both defined in a neighbourhood $U$ of $0$ in $\C$
  and both are real along $\R$,
  such that $w=Bf$ holds on $U\cap\set{\im(z)\geq0}$.
  In particular there exist $a,b\in\R^n\setminus\set{0}$
  such that
  \begin{gather*}
    u(z)=az+o(\abs{z})
    \quad\text{and}\quad
    v(z)=bz+o(\abs{z})
  \end{gather*}
  as $z$ tends to $0$ in $\D^+$.
  
  After composing with a complex linear transformation $A\in\gl(\C^n)$ such that $A(\R^n)=\R^n$
  we can further assume that $a=(1,0,\ldots,0)^T$ and $b=(0,\ldots,0,1)^T$ as column vectors in $\C^n$.
  Therefore, we have that the coordinate functions $\varphi\define u_1$ and $\psi\define v_n$
  are local diffeomorphisms about $\varphi(0)=0$ and $\psi(0)=0$, resp.
  Hence, by Taylor's formula we get
                                                            \begin{gather*}
                                                              u\circ\varphi^{-1}(z)=
                                                              \big(z,o(\abs{z})\big)
                                                              \qquad\text{in}\;\;
                                                              \C\times\C^{n-1},
                                                              \\
                                                              v\circ\psi^{-1}(z)=
                                                              \big(o(\abs{z}),z\big)
                                                              \qquad\text{in}\;\;
                                                              \C^{n-1}\times\C.
                                                            \end{gather*}
  Recall that we have assumed $u(z_{\nu})=v(w_{\nu})$ for our sequences $z_{\nu}\ra 0$ and $w_{\nu}\ra 0$.
  So setting $\tilde{z}_{\nu}\define\varphi(z_{\nu})$ and $\tilde{w}_{\nu}\define\psi(w_{\nu})$
  for all $\nu\in\N$ we find
                                                            \begin{gather*}
                                                              \tilde{z}_{\nu}=o(\abs{\tilde{w}_{\nu}})
                                                              \quad\text{and}\quad
                                                              \tilde{w}_{\nu}=o(\abs{\tilde{z}_{\nu}})
                                                            \end{gather*}
  for all $\nu\in\N$.
  But this implies
                                                            \begin{gather*}
                                                              \Babs{\frac{\tilde{z}_{\nu}}{\tilde{w}_{\nu}}}=
                                                              \Babs{\frac{o(\abs{\tilde{w}_{\nu}})}{\abs{\tilde{w}_{\nu}}}}
                                                              \ra0
                                                              \quad\text{and}\quad
                                                              \Babs{\frac{\tilde{w}_{\nu}}{\tilde{z}_{\nu}}}=
                                                              \Babs{\frac{o(\abs{\tilde{z}_{\nu}})}{\abs{\tilde{z}_{\nu}}}}
                                                              \ra0
                                                            \end{gather*}
  as $\nu$ tends to infinity.
  Consequently,
                                                            \begin{gather*}
                                                              1=
                                                              \frac{\tilde{z}_{\nu}}{\tilde{w}_{\nu}}
                                                              \frac{\tilde{w}_{\nu}}{\tilde{z}_{\nu}}
                                                              \ra0
                                                            \end{gather*}
  as $\nu\ra\infty$, which is a contradiction.
  This shows that $u$ and $v$ must be tangent at $0$
  which proves the claim.
\end{proof}

Lemma \ref{signofdelta} allows us to define the {\bf sign of the boundary intersection}
at the accumulation point $(0,0)$ to be the sign of the real number $\delta$,
i.e.\
                                                            \begin{gather}
                                                              \label{signofinter}
                                                              \delta_{u,v}\define\sign(\delta).
                                                            \end{gather}
Notice that this definition is intrinsic, i.e.\ is invariant under conformal reparametrisations
which respect the boundary condition.
Now it turns out that under the assumptions of the proceeding section
the images of the half-disc maps $u$ and $v$ coincide,
provided that the sign is positive.
This is to say that the unique continuation principle holds true.

\begin{lem}
  \label{isolembb}
  If the sign $\delta_{u,v}$ is positive
  then there exist open neighbourhoods $U$ and $V$ of $0$ in $\D^+$
  such that $u(U)=v(V)$.
  In fact, there exists an analytic diffeomorphism $\varphi$ between pointed neighbourhoods of $0$ in $\C$
  such that $v=u\circ\varphi$ hold true on $V$.
\end{lem}

In order to prove this lemma
we will need a particular coordinate system about the intersection point
making one of the holomorphic discs flat.
The following result can be obtained as in \cite[Lemma 2.4.2]{mcdsal04}:

\begin{lem}
  \label{nicechart}
  Let $w:(\D^+,[-1,1],0)\ra(M,L,p)$ be a $J$-holomorphic embedding.
  Then there exist a neighbourhood $U$ of $0$ in $\C^n$
  and an embedding
                                                            \begin{gather*}
                                                              \Phi:(U,U\cap\R^n,0)\lra(M,L,p)
                                                            \end{gather*}
  such that
                                                            \begin{gather*}
                                                              \Phi^{-1}\circ w(z)=(z,0,\ldots,0)
                                                              \quad\text{in}\;\;
                                                              \C\times\C^{n-1}
                                                            \end{gather*}
  for all $z\in\D^+$ and $\tilde{J}\define\Phi^*J= d\Phi^{-1}\circ J\circ d\Phi$ satisfies
                                                            \begin{gather}
                                                              \label{jisstandard}
                                                              \tilde{J}(z,0)=i
                                                            \end{gather}
  for all $z\in\D^+$.
\end{lem}

\begin{proof}[{\it Proof of Lemma \ref{isolembb}}]
  We will follow the line reasoning from \cite[p.~90]{hof99}.
  By Lemma \ref{nicechart} we can assume that our $J$-holomorphic half-disc maps $u,v$ are given by
                                                            \begin{gather*}
                                                              u,v:(\D^+,[-1,1],0)\lra(\C^n,\R^n,0)\\
                                                              u=(u_1,\ldots,u_n),v=(v_1,\ldots,v_n),
                                                            \end{gather*}
  such that $u(z)=(z,0,\ldots,0)$ for all $z\in\D^+$,
  where $J$ is an almost complex structure on $\C^n$ such that $J=i$ on $\D^+\times\set{0}$.
  By our assumption there is a positive real number $\delta>0$
  such that $du(0)=\delta dv(0)$.
  So we find that
                                                            \begin{gather*}
                                                              \partial_tv_1(0)=
                                                              \tfrac{1}{\delta}\partial_tu_1(0)=
                                                              \tfrac{1}{\delta}i
                                                            \end{gather*}
  has positive imaginary part.
  W.l.o.g. we can assume that $\im\big(\partial_tv_1(z)\big)>0$ for all $z\in\D^+$.
  Therefore, the function $t\mapsto\im\big(v_1(s+it)\big)$ is monotone increasing for all $s\in[-1,1]$
  and, hence, $v_1(z)\in\set{\im(z)>0}$ for all $z\in\set{\im(z)>0}\cap\D^+$.
  Consequently, we have that
                                                            \begin{gather}
                                                              \label{v1standard}
                                                              J(v_1,0)=i.
                                                            \end{gather}
  If we write $\tilde{v}$ for the tuple $(v_2,\ldots,v_n)$
  in the following
  we get that
                                                            \begin{gather*}
                                                              \int_0^1D_2J(v_1,\tau\tilde{v})\set{\tilde{v}}d\tau=
                                                              \int_0^1\partial_{\tau}J(v_1,\tau\tilde{v})d\tau=
                                                              J(v_1,\tilde{v})-J(v_1,0)=
                                                              J(v)-i
                                                            \end{gather*}
  on $\D^+$.
  By $J$-holomorphicity we get that
                                                            \begin{gather}
                                                              \label{cspforbwoproj}
                                                              0=
                                                              \partial_sv+J(v)\partial_tv=
                                                              \partial_sv+i\partial_tv+B\tilde{v},
                                                            \end{gather}
  where
                                                            \begin{gather*}
                                                              B\define
                                                              \left(
                                                                \int_0^1D_2J(v_1,\tau\tilde{v})\set{\,.\,}d\tau
                                                                \cdot\partial_tv
                                                              \right)
                                                            \end{gather*}
  is an element of $C^{\infty}\big(\D^+,\Hom_{\R}(\C^{n-1},\C^n)\big)$.
  Setting $\tilde{B}\define\pr_{\C^{n-1}}B$, which yields an element in $C^{\infty}\big(\D^+,\Ed_{\R}(\C^{n-1})\big)$,
  we get finally
                                                            \begin{gather*}
                                                              0=
                                                              \partial_s\tilde{v}+i\partial_t\tilde{v}+\tilde{B}\tilde{v}.
                                                            \end{gather*}
  The relative Carleman similarity principle (see \cite[Theorem A.2]{abb04}) implies
  that there exist $0<\varepsilon\leq1$,
  a complex analytic function $f:\big(\D_{\varepsilon},[-\varepsilon,\varepsilon],0\big)\ra\big(\C^{n-1},\R^{n-1},0\big)$,
  and $C\in C^0\big(\D_{\varepsilon},\gl_{\R}(\C_{n-1})\big)$,
  such that
                                                            \begin{gather*}
                                                              \tilde{v}(z)=C(z)f(z)
                                                            \end{gather*}
  for all $z \in \D_{\varepsilon}^+$.
  By assumption we have that $\tilde{v}(w_{\nu})=0$ for all $\nu\in\N$, where
  $w_{\nu}\ra 0$ in $\D^+\setminus\set{0}$.
  This implies $f(w_{\nu})=0$ for all $\nu\in\N$ and, hence, $f=0$.
  Consequently, $\tilde{v}=0$.
  Thus there exists an open neighbourhood $V$ of $0$ such that
  $\varphi\define v_1$ is analytic on the interior of $V$,
  see \eqref{cspforbwoproj},
  and sends real numbers to real numbers.
  Moreover, $v=u\circ\varphi$ on $V$.
  Setting $U=\varphi(V)$ this gives the claim.
\end{proof}

%%%%%%%%%%%%%%%%%%%%%%%%%%%%%%%%%%%%%%%%%%%%%%%%%%%%%%%%%%%%%%%%%%%%%%
%%%%%%%%%%%%%%%%%%%%%%%%%%%%%%%%%%%%%%%%%%%%%%%%%%%%%%%%%%%%%%%%%%%%%%

\section{Mixed intersections\label{section:mixed}}

The aim of this section is to finish with the proof of Theorem \ref{annulusproperty}
which says that a simple holomorphic disc has the annulus property.
First of all we will use Lemma \ref{isolembb} to show
that the set of accumulation points $R_1$
can not contain a subset of the form $c([-1,1])$ for any embedding $c:[-1,1]\ra\A$.

\begin{proof}[{\it Proof of Lemma \ref{noembeddingintor1}}]
  Arguing by contradiction we suppose there is an embedding $c:[-1,1]\ra\A$ with $c([-1,1])\subset R_1$.
  Set $c(0)=z^*$ and take a positive real number $\delta$ with $B_{\delta}(z^*)\subset B^*_{\varepsilon}(z_0)$.
  By shrinking $\delta>0$ we can assume that $B_{\delta}(z^*)\setminus c([-1,1])$
  has exactly two components $\Sigma_1$ and $\Sigma_2$.
  By uniformisation
  there are analytic embeddings
                                                            \begin{gather*}
                                                              \varphi_j:\big(\D^+,\partial\D^+,0\big)
                                                              \lra
                                                              \big(\Sigma_j,c([-1,1]),z^*\big),
                                                            \end{gather*}
  one for each $j=1,2$.
  Notice that the images of the vector $\partial_s$
  under the differentials $d\varphi_1$ and $d\varphi_2$ point in opposite directions in $T_{z^*}c([-1,1])$.
  Moreover, by assumption the $J$-holomorphic half-discs
                                                            \begin{gather*}
                                                              u_j\define u\circ\varphi_j:
                                                              \big(\D^+,\partial\D^+\big)
                                                              \lra
                                                              (M,L)
                                                            \end{gather*}
  intersect the given $J$-holomorphic disc $u$ along the boundary,
  i.e.\
                                                            \begin{gather*}
                                                              u_j(\partial\D^+)\subset
                                                              u(\partial\D).
                                                            \end{gather*}
  Because $u\restrict_{B^*_{\varepsilon}(z_0)}$ is an embedding and $z^*\notin\crt(u)$
  we may assume (after shrinking $\delta>0$ again if necessary)
  that this is a local intersection of $3$ embedded $J$-holomorphic half-discs
  along their boundaries.
  In particular,
  for one $j=1,2$ the sign of boundary intersection $\delta_{u,u_j}$ must be positive.
  So we derived from Lemma \ref{isolembb} that $2$ of the branches of $u$ along the boundary must overlap.
  But this means that $u$ can not be simple.
  This is a contradiction.
\end{proof}

As we already remarked we obtained in fact that
the compact set $S_1$ is nowhere dense in $\A$ and
does not contain embedded arcs.
While the set of interior intersections
                                                            \begin{gather*}
                                                              S_1^{\inter}\define
                                                              \pr_1\!\big(S_{\inter}(u)\big)\cap\A=
                                                              \bigset{
                                                                x\in\A
                                                                \suchthatb
                                                                \exists
                                                                y\in B_1(0):
                                                                u(x)=u(y)
                                                              }
                                                            \end{gather*}
is discrete it seems that the set of {\bf mixed self-intersection points of} $u$
                                                            \begin{gather*}
                                                              S_{\mix}(u)\define
                                                              S(u)\cap
                                                              \big(B_1(0)\times\partial\D\big)
                                                            \end{gather*}
is rather complicated.
In order to understand the topology of the set
                                                            \begin{gather*}
                                                              S_1=S_1^{\mix}\sqcup S_1^{\inter},
                                                            \end{gather*}
where
                                                            \begin{gather*}
                                                              S_1^{\mix}\define
                                                              \pr_1\!\big(S_{\mix}(u)\big)\cap\A=
                                                              \bigset{
                                                                x\in\A
                                                                \suchthatb
                                                                \exists
                                                                y\in\partial\D:
                                                                u(x)=u(y)
                                                              },
                                                            \end{gather*}
we need more information about the mixed self-intersections.
  
First of all
we will extract an important subclass of accumulations of self-intersection points.
Namely, we consider the set of  {\bf mixed self-intersections of virtual boundary type}
                                                            \begin{gather*}
                                                              \RR^{\virt}=
                                                              \acc\big(S(u)\cap\A\times\partial\D\big)
                                                            \end{gather*}
and denote the projection to the first coordinate by 
                                                            \begin{gather*}
                                                              R_1^{\virt}\define\pr_1\!\RR^{\virt}.
                                                            \end{gather*}
Observe that $R_1^{\virt}=\acc(S_1^{\mix})$.
Its complement,
the set of accumulation points which only can be reached from the interior,
is denoted by
                                                            \begin{gather*}
                                                              \tilde{R}_1\define R_1\setminus R_1^{\virt}.
                                                            \end{gather*}
The reason for these definitions is the following key property:
By \eqref{disjointunionofintpoints} we have $R_1\subset S_1^{\mix}$ which implies
                                                            \begin{gather}
                                                              \label{keyproperty}
                                                              \acc(R_1)\subset R_1^{\virt}.
                                                            \end{gather}
In order to prove the theorem
we are going to construct a covering of $S_1$
using the decomposition of the accumulation points
                                                            \begin{gather*}
                                                               R_1=\tilde{R}_1\sqcup R_1^{\virt}
                                                            \end{gather*}
into $\tilde{R}_1$ and $R_1^{\virt}$.

\begin{lem}
  \label{disjboundlem}
  Consider a point $z^*\in R_1^{\virt}$.
  Then for any $\rho>0$ there exists $\delta\in(0,\rho)$ such that
                                                            \begin{gather*}
                                                              S_1\cap\partial B_{\delta}(z^*)=
                                                              \emptyset.
                                                            \end{gather*}
\end{lem}

In order to prove this fact we consider $J$-holomorphic embeddings 
                                                            \begin{gather*}
                                                              u:\big(B_1(0),0\big)\lra(M,p)
                                                              \quad\text{and}\quad
                                                              v:\big(\D^+,[-1,1],0\big)\lra(M,L,p)
                                                            \end{gather*}
such that there are sequences $z_{\nu}\ra 0$ in $B_1(0)\setminus\set{0}$
and $w_{\nu}\ra 0$ in $\partial\D^+\setminus\set{0}$ of complex numbers
such that $u(z_{\nu})=v(w_{\nu})$ for all $\nu\in\N$.
We remark that the condition on the points $w_{\nu}$ to lie on the boundary can be interpreted as a
{\bf virtual boundary condition}.

\begin{lem}
  \label{typetwo}
  For all $\varepsilon>0$
  there exist $\delta\in (0,\varepsilon)$ and an embedding
                                                            \begin{gather*}
                                                              c:\big([-1,1],0\big)\lra\big(\D_{\delta},0\big),
                                                            \end{gather*}
  where $\D_{\delta}\define\overline{B_{\delta}(0)}$,
  such that the following conditions are satisfied:
  \begin{itemize}
  \item [(i)]
    the points $z_{\nu}$ lie on the image $\gamma\define\Image(c)$
    provided $z_{\nu}\in\D_{\delta}$ for all $\nu\in\N$,
  \item [(ii)]
    the intersection $\gamma\cap\partial B_{\delta}(0)$ equals $\set{c(\pm1)}$ and is transverse,
  \item [(iii)]
    each point $\zeta\in\D_{\delta}$ with $u(\zeta)\in L$
    lies in fact on $\gamma$.
  \end{itemize}
\end{lem}

\begin{proof}
  By Lemma \ref{nicechart} we can assume that $v(z)=(z,0,\ldots,0)$ in $\C\times\C^{n-1}$ for all $z\in\D^+$
  and that $u(z_{\nu})=v(w_{\nu})\in\R^n$.
  Writing $u=(u_1,\tilde{u})$ with respect to the splitting $\C\times\C^{n-1}$ we see
  that $u_1(z_{\nu})\in\R$ and $\tilde{u}(z_{\nu})=0$ for all $\nu\in\N$.
  Hence, $d\tilde{u}(0)=0$.
  Therefore, we can assume that the restricted map
  $u_1:(\D_{\delta},0)\ra (\C,0)$ is an embedding for some $\delta>0$.
  Therefore, in a neighbourhood of zero we can invert $u_1$,
  i.e.\ $\varphi=u_1^{-1}$ exists.
  By shrinking $\delta>0$ again if necessary,
  we can assume that $\R\cap u_1(\D_{\delta})$ is connected
  with transverse intersections at the boundary.
  We get the first claim by considering $c\define\varphi\restrict_{\R}$.
  
  By making $\delta>0$ smaller if necessary we can assume that $u(\D_{\delta})$
  is contained in the domain $(V,p)$
  of that chart map we used to flatten $v$, see Lemma \ref{nicechart}.
  We can assume that $V\cap L$ is connected, i.e.\ in the chart we have $L=\R^n$.
  Hence, $\zeta\in\D_{\delta}$ with $u(\zeta)\in L$ implies $u_1(\zeta)\in\R$
  and, therefore, $\zeta=\varphi\circ u_1(\zeta)\in\gamma$.
  This proves the third claim.
\end{proof}

  \begin{proof}[{\it Proof of Lemma \ref{disjboundlem}}]
    By Proposition \ref{pointsfinite} the cardinality of the set $u^{-1}\big(u(z^*)\big)\cap\partial\D$ is finite,
    say $k\in\N$.
    If we compose our holomorphic map $u$
    with local parametrisations of $\D$ about the $k$ intersection points
    we find holomorphic maps $v_1,\ldots,v_k$ defined on half-discs
    to which we can apply Lemma \ref{typetwo}.
    So we find $\delta>0$ as small as we wish
    and an embeddings
                                                            \begin{gather*}
                                                              c_1,\ldots,c_k:
                                                              \big([-1,1],0\big)
                                                              \lra
                                                              \big(\overline{B_{\delta}(z^*)},z^*\big),
                                                            \end{gather*}
    such that (for each $j=1,\ldots,k$) the intersection of the image $\gamma_j$ of $c_j$
    with the boundary $\partial B_{\delta}(z^*)$ is transverse,
    equals the set of endpoints $\set{c_j(\pm1)}$,
    and
    \[
    S_1^{\mix}\cap\overline{B_{\delta}(z^*)}
    \subset
    \big(\gamma_1\cup\ldots\cup\gamma_k\big)
    \enifed
    \Gamma.
    \]
    In other words we get
    \begin{gather}
      \label{onedimsation2}
      S_1^{\mix}\cap\overline{B_{\delta}(z^*)}=
      S_1^{\mix}\cap\Gamma.
    \end{gather}
    Now, by Lemma \ref{noembeddingintor1} the sets $S_1^{\mix}\cap\gamma_j$
    contain no interior point.
    We conclude that $\gamma_j\setminus S_1^{\mix}$
    is an open and dense subset of $\gamma_j$ for any $j=1,\ldots,k$.
    Therefore, we find $\delta>0$ and $\delta_1\in (0,\delta)$ such that
    \begin{gather*}
      S_1^{\mix}\cap\A_{\delta_1,\delta}(z^*)=
      \emptyset
    \end{gather*}
    if we write
    \begin{gather*}
      \A_{\delta_1,\delta}(z^*)\define
      \overline{B_{\delta}(z^*)}\setminus B_{\delta_1}(z^*)
    \end{gather*}
    for the standard annulus.
    Moreover, Theorem \ref{mwmagerthm} and \eqref{onedimsation2} imply that
    \begin{gather*}
      \acc\big(S_1^{\inter}\cap B_{\delta}(z^*)\big)\subset
      S_1^{\mix}\cap\Gamma.
    \end{gather*}
    Hence,
    \begin{gather*}
      S_1^{\inter}\cap\A_{\delta_1,\delta}(z^*)
    \end{gather*}
    is a finite set.
    Consequently, we find $\delta_2\in (\delta_1,\delta)$ so that
    \begin{gather*}
      S_1^{\inter}\cap\partial B_{\delta_2}(z^*)=
      \emptyset.
    \end{gather*}
    Taking $\delta_2$ instead of $\delta>0$ yields the claim.
  \end{proof}

  \begin{lem}
    \label{disjboundlem2}
    Consider a point $z^*\in\tilde{R}_1$.
    Then for any $\rho>0$ there exists $\delta\in(0,\rho)$ such that
    $S_1\cap\partial B_{\delta}(z^*)=\emptyset$.
  \end{lem}

  \begin{proof}
    Because of \eqref{keyproperty} and Lemma \ref{disjboundlem} we can consider the case,
    where $z^*$ is isolated in $R_1$.
    This means, that we can find $\delta>0$ such that
    $R_1\cap\overline{B_{\delta}(z^*)}=\set{z^*}$.
    With Theorem \ref{mwmagerthm} we see, that
    $S_1\cap\A_{\delta/2,\delta}(z^*)$
    is a finite set.
    This proves the claim.
  \end{proof}

  \begin{lem}
    Consider a point $z^*\in S_1$.
    Then for any $\rho>0$ there exists $\delta\in(0,\rho)$ such that
    $S_1\cap\partial B_{\delta}(z^*)=\emptyset$.
  \end{lem}

  \begin{proof}
    Because of $\acc(S_1)=\tilde{R}_1\sqcup R_1^{\virt}$,
    Lemma \ref{disjboundlem} and Lemma \ref{disjboundlem2}
    we can consider the case,
    where $z^*$ is isolated in $S_1$.
    This means, that we can find $\delta>0$ such that
    $S_1\cap\overline{B_{\delta}(z^*)}=\set{z^*}$.
    This proves the claim.
  \end{proof}

\begin{proof}[{\bf Proof of Theorem \ref{annulusproperty} (part II)}]
  Summing up, for each $z^*\in S_1$ there exists $\delta>0$ such that
  \begin{gather}
    \label{disjboundglob}
    S_1\cap\partial B_{\delta}(z^*)=\emptyset
    \quad\text{and}\quad
    \partial\A\cap B_{\delta}(z^*)
    \;\;\text{is connected.}
  \end{gather}
  By compactness of $S_1$ we find $N\in\N$, $\delta_1,\ldots,\delta_N>0$ and $z_1,\ldots,z_N$ in $S_1$
  such that
                                                            \begin{gather*}
                                                              S_1\subset \bigcup_{j=1}^NB_{\delta_j}(z_j)
                                                            \end{gather*}
  and \eqref{disjboundglob} holds with $z^*$ replaced by each of the $z_1,\ldots,z_N$.
  Notice, that by openness of $\A\setminus S_1$ we find for each $j=1,\ldots,N$ some $\tilde{\delta}_j\in (0,\delta_j)$
  keeping the mentioned properties in \eqref{disjboundglob}.
  By removing the annuli $B_{\delta_j}(z_j)\setminus\overline{B_{\tilde{\delta}_j}(z_j)}$
  we find $K\in\N$ and pairwise disjoint, closed, simply connected subsets $B_1,\ldots,B_K$ of $\A$
  (with non-empty interior)
  covering $S_1$ such that $\partial\A\cap B_j$ is connected for all $j=1,\ldots,K$.
  Consequently, $\A\setminus (B_1\cup\ldots\cup B_K)$ contains an embedded closed curve $\gamma$
  with winding number equal to $1$ with respect to $z_0$.
  A tubular neighbourhood of $\gamma$ yields an annulus $A_{\varepsilon,z_0}$
  with $u^{-1}\big(u(A_{\varepsilon,z_0})\big)=A_{\varepsilon,z_0}$ as desired.
  This finishes the proof of Theorem \ref{annulusproperty}.
\end{proof}

\begin{rem}
  One can show that $R_1$ is totally path disconnected and that $R_1$ has $2$-dimensional
  Lebesgue measure equal to zero.
\end{rem}

%%%%%%%%%%%%%%%%%%%%%%%%%%%%%%%%%%%%%%%%%%%%%%%%%%%%%%%%%%%%%%%%%%%%%%
%%%%%%%%%%%%%%%%%%%%%%%%%%%%%%%%%%%%%%%%%%%%%%%%%%%%%%%%%%%%%%%%%%%%%%

\section{The half-annulus property\label{halfannuluspropertysubsec}}

Once shown the annulus property for simple holomorphic discs
one naturally asks for the situation at boundary points.
Again we consider a $J$-holomorphic disc $u:(\D,\partial\D)\ra(M,L)$
in an almost complex manifold $(M,J)$
which takes boundary values in a totally real submanifold $L$ of $M$.
The content of this section is a version of Theorem \ref{annulusproperty}
valid for boundary points of those holomorphic discs $u$.

\begin{defi}
  \label{def:halfannulusprop}
  We will say that $u$ has the {\bf half-annulus property}
  if for any point $z_0\in\partial\D$ and any
  $\varepsilon>0$
  there exists a closed subset $A_{\varepsilon,z_0}$
  diffeomorphic to the annulus $S^1\times [0,1]$
  such that the following conditions are satisfied:
  \begin{enumerate}
  \item [(i)]
    $A_{\varepsilon,z_0}\subset B^*_{\varepsilon}(z_0)$ (viewed as subsets of $\C$),
    such that $A_{\varepsilon,z_0}\cap\partial\D$ has exactly two components,
  \item [(ii)]
    a boundary component of $A_{\varepsilon,z_0}$ has winding number $1$ around $z_0$,
  \item [(iii)]
    $u\restrict_{A_{\varepsilon,z_0}^+}$ is an embedding,
    where $A_{\varepsilon,z_0}^+\define A_{\varepsilon,z_0}\cap\D$, and
  \item [(iv)]
    $u^{-1}\big(u(A_{\varepsilon,z_0}^+)\big)=A_{\varepsilon,z_0}^+$.
  \end{enumerate}
\end{defi}

Not every simple holomorphic disc has the half-annulus property.
For example the map $u:(\D,\partial\D)\ra(\C P^1,\R P^1)$
induced by $z\mapsto z^2$ on the upper half-plane
is simple;
but none of the boundary points is injective.
In view of the last two conditions this holomorphic disc can not have the half-annulus property.
But there is an interesting subclass of simple holomorphic discs
which will have the half-annulus property.

\begin{defi}
  \label{defofstrparsimple}
  A $J$-holomorphic disc $u$ is {\bf strongly simple along the boundary}
  or {\bf strongly} $\partial$-{\bf simple}
  if $\inj(u)\cap\partial\D$ is dense in $\partial\D$.
\end{defi}

As we will see in Proposition \ref{frombdrtoint} below
all $\partial$-simple $J$-holomorphic discs are simple.
But there are simple $J$-holomorphic discs
which can not be simple along the boundary
as the above example $u:(\D,\partial\D)\ra(\C P^1,\R P^1)$ shows.
So we introduced a more restrictive notion.

Moreover,
any $J$-holomorphic disc $u$
which has the half-annulus property
is $\partial$-simple.
This is because for all $z_0\in\partial\D$ and for all $\varepsilon>0$
the half-annulus $A_{\varepsilon,z_0}^+$ is a subset of $\inj(u)$.
So any point in the boundary $\partial\D$ can be approximated by those injective points of $u$
which are boundary points at the same time.
In fact we have the following:

\begin{thm}
  \label{halfannulusproperty}
  A $J$-holomorphic disc is strongly $\partial$-simple if and only if it has the half-annulus property.
\end{thm}

\begin{proof}
  We will show that
  any $\partial$-simple $J$-holomorphic disc $u$ has the half-annulus property.
  This of course is enough to prove the theorem.
  To establish the half-annulus property of $u$
  we will reduce the proof to the one of Theorem \ref{annulusproperty}.

  For that we consider a point $z_0\in\partial\D$ and take $\varepsilon>0$
  such that $\D\setminus\partial B_{\varepsilon}(z_0)$ is disconnected.
  By Proposition \ref{pointsfinite} %\eqref{precritpointsfinite}
  we can choose $\varepsilon>0$ so small such that
                                                            \begin{gather}
                                                              \label{halffirstconclusion}
                                                              B^*_{\varepsilon}(z_0)\cap\crt(u)=\emptyset.
                                                            \end{gather}
  In particular, $u\restrict_{B^*_{\varepsilon}(z_0)}$ is an immersion but potentially with self-intersections.
  So it remains to discuss the double points.

  First of all we remark
  that by Proposition \ref{injuopen}
  the set $\inj(u)\cap\partial\D$ is open in $\partial\D$.
  By the assumed strong $\partial$-simplicity of $u$ it is dense in $\partial\D$ as well.
  Therefore, we find $0<\varepsilon_1<\varepsilon_2<\varepsilon$ such that
  $\A_{\varepsilon_1,\varepsilon_2}(z_0)\cap\partial\D$ is a subset of $\inj(u)$,
  where we set
  \[
  \A_{\varepsilon_1,\varepsilon_2}(z_0)\define\set{\varepsilon_1\leq\abs{z-z_0}\leq\varepsilon_2}.
  \]
  In fact we find $\varrho\in(0,1)$ such that
  for the neighbourhood $\A_{\varrho,1}(0)$ of $\partial\D$ we have
                                                            \begin{gather}
                                                              \label{reductiontononhalf}
                                                              \A_{\varepsilon_1,\varepsilon_2}(z_0)\cap\A_{\varrho,1}(0)
                                                              \subset\inj(u).
                                                            \end{gather}
  This can be seen by a further application of Proposition \ref{injuopen}.

  On the other hand,
  by Theorem \ref{mwmagerthm},
  the set
  $S(u)\cap\big(B_{\varrho}(0)\times B_{\varrho}(0)\big)$ is finite.
  Hence, we can take $0<\varepsilon_1<\varepsilon_2<\varepsilon$ so small
  such that
                                                            \begin{gather*}
                                                              S(u)\cap
                                                              \big(
                                                                \A_{\varepsilon_1,\varepsilon_2}(z_0)\times\A_{\varepsilon_1,\varepsilon_2}(z_0)
                                                              \big)=
                                                              \emptyset.
                                                            \end{gather*}
  In other words $u\restrict_{\A_{\varepsilon_1,\varepsilon_2}(z_0)}$ is injective, and, therefore, an embedding.
  In addition observe that by \eqref{reductiontononhalf}
                                                            \begin{gather*}
                                                              u^{-1}\Big(u\big(\A_{\varepsilon_1,\varepsilon_2}(z_0)\cap\A_{\varrho,1}(0)\big)\Big)
                                                              =\A_{\varepsilon_1,\varepsilon_2}(z_0)\cap\A_{\varrho,1}(0).
                                                            \end{gather*}
  So it remains to study the subset $u^{-1}\big(u(\A)\big)$ of $\D$,
  where
                                                            \begin{gather*}
                                                              \A\define
                                                              \A_{\varepsilon_1,\varepsilon_2}(z_0)\cap
                                                              \overline{B_{\varrho}(0)}.
                                                            \end{gather*}
  This can be done by similar arguments as used in Theorem \ref{annulusproperty}.
\end{proof}

As we already remarked, the half-annulus property (or equivalently strong $\partial$-simplicity)
implies simplicity.
Using the fundamental work of Lazzarini \cite{lazz00}
this can be seen as follows:

\begin{prop}
  \label{frombdrtoint}
  Any strongly $\partial$-simple holomorphic disc is simple.
\end{prop}

\begin{proof}
  Arguing by contradiction
  we suppose that the $J$-holomorphic disc $u$ is not simple.
  Then there exist disjoint open sets $B_1$ and $B_2$ in $B_1(0)\setminus\crt(u)$
  such that $u(B_1)=u(B_2)$, see Proposition \ref{equivcharofsimple}.
  Denote by $U_1$ and $U_2$ the open subsets of $B_1(0)\setminus\crt(u)$
  such that $B_1\subset U_1$, $B_2\subset U_2$ and $u(U_1)=u(U_2)$,
  which are maximal in the following sense:
  if there are open subsets $V_1$ and $V_2$ of $B_1(0)\setminus\crt(u)$
  with $U_1\subset V_1$, $U_2\subset V_2$ and $u(V_1)=u(V_2)$
  then we have $U_1=V_1$ and $U_2=V_2$.
  The existence is insured by Zorn's Lemma.
  
  We consider a point $(z_1,z_2)\in\partial U_1\times\partial U_2$,
  where $\partial U_j$ denotes the set-theoretical boundary of $U_j$ (for $j=1,2$).
  Additionally, we suppose that $du(z_j)\neq 0$ for $j=1,2$.
  In particular we have $z_1\neq z_2$.
  Notice, by Lemma \ref{weakversion}
  (recalling that any analytic function is open)
  and maximality of the $U_j$ the case $(z_1,z_2)\in B_1(0)\times B_1(0)$ is excluded.
  So by symmetry there are two cases left.
  
  First of all we assume that $(z_1,z_2)\in\partial\D\times B_1(0)$.
  By the construction of the $U_j$ and \cite[Proposition 4.5]{lazz00}
  we find open subsets $V_1$ of $\D$ and $V_2$ of $B_1(0)$
  such that $z_j\in V_j$ for $j=1,2$ and $u(V_1)\subset u(V_2)$.
  Hence, $V_1\cap\partial\D$ is contained in $\partial\D\setminus\inj(u)$,
  which is not possible by strong $\partial$-simplicity of $u$.
  
  We consider the case $(z_1,z_2)\in\partial\D\times\partial\D$.
  Because of $(z_1,z_2)\in\acc\big(S_{\inter}(u)\big)$ there exists an non-zero real number $\delta$
  such that $\partial_ru(z_1)=\delta\partial_ru(z_2)$
  (see Lemma \ref{signofdelta}),
  where $\partial_r$ denotes the radial derivative.
  By the proof of Lemma \ref{isolembb} this number $\delta$ must be positive,
  because in this particular situation the accumulation point is approached from inside.
  So we see (again with Lemma \ref{isolembb})
  there are open and disjoint subsets $I_1$ and $I_2$ of $\partial\D$
  such that $u(I_1)=u(I_2)$.
  But this contradicts the fact that $u$ is strongly $\partial$-simple.
  
  Consequently, we infer that $\partial U_j\subset\crt(u)$,
  which is a finite set.
  But this is not possible, because the cardinality of the boundary of any open and bounded subset of $\R^2$
  is infinite.
  (Indeed, any ray starting from a point in $U_j$
  going through a point of a circle in the complement of $U_j$ contains a boundary point.)
  Therefore, $u$ must be simple.
\end{proof}

%%%%%%%%%%%%%%%%%%%%%%%%%%%%%%%%%%%%%%%%%%%%%%%%%%%%%%%%%%%%%%%%%%%%%%
%%%%%%%%%%%%%%%%%%%%%%%%%%%%%%%%%%%%%%%%%%%%%%%%%%%%%%%%%%%%%%%%%%%%%%

\section[Injective points along the boundary]{Injective points
  along the boundary\label{injpointsalongtheboundary}}

In this section we discuss a second notion for a holomorphic disc to be injective on boundary points.
This allows us to verify the half-annulus property in many cases, as we will see in the next sections.
To motivate the definition we consider the holomorphic disc $(\D,\partial\D)\ra(\C P^1,\R P^1)$
defined by $z\mapsto z^3$ on the upper half-plane.
Restricted to the boundary this map is injective and immersive except at $0$ and $\infty$.
But no boundary point is injective in the sense of \eqref{sipoint}.

To study this phenomenon more systematically
we consider a $J$-holomorphic discs $u:(\D,\partial\D)\ra(M,L)$
for an almost complex structure $J$ as usual.
Following Oh \cite{oh96} we call a point $z\in\partial\D$ satisfying
                                                            \begin{gather}
                                                              \label{sipointboud}
                                                              u^{-1}\big(u(z)\big)\cap\partial\D=\set{z}
                                                              \quad\text{and}\quad
                                                              du(z)\neq0
                                                            \end{gather}
an {\bf injective point of} $u\bd$.
The set of all injective points of $u\bd$ is denoted by
                                                            \begin{gather*}
                                                              \inj(u\bd)\define
                                                              \bigset{
                                                                z\in\partial\D
                                                                \suchthatb
                                                                z\;\;
                                                                \text{is an injective point of}
                                                                \;\;u\bd
                                                              }.
                                                            \end{gather*}
We remark that
any injective point of $u$ on the boundary is an injective point of $u\bd$,
i.e.\
                                                            \begin{gather}
                                                              \label{weakvsstrong}
                                                              \inj(u)\cap\partial\D\subset\inj(u\bd).
                                                            \end{gather}
But in general equality does not hold as the above example shows.

Similar to Proposition \ref{injuopen} we have:

\begin{prop}
  \label{injuopenbound}
  The set $\inj(u\bd)$ is open in $\partial\D$.
\end{prop}

A $J$-holomorphic disc $u$ is called {\bf somewhere injective along the boundary}
if $u\bd$ has an injective point, i.e.\ if $\inj(u\bd)\neq\emptyset$.
Similar to Section \ref{section:characterisation}
this notion is related to
the set of self-intersection points
                                                            \begin{gather*}
                                                              S_{\partial}(u)\define
                                                              S(u)\cap\big(\partial\D\times\partial\D\big)
                                                            \end{gather*}
of $u\bd$, which we call the set of {\bf boundary-self-intersection points of} $u$.
We denote by
                                                            \begin{gather*}
                                                              S_1^{\partial}(u)\define\pr_1S_{\partial}(u)
                                                            \end{gather*}
the projection to the first coordinate.
As in the non-boundary case we have
                                                            \begin{gather}
                                                              \label{clousureofsubdry}
                                                              \overline{S_{\partial}(u)}\subset S_{\partial}(u)\sqcup\Delta_{\crit(u\bd)},
                                                            \end{gather}
because a singular point of $u\bd$ is singular, i.e.\
                                                            \begin{gather*}
                                                              \crit(u)\cap\partial\D=\crit(u\bd).
                                                            \end{gather*}
In general, the inclusion in \eqref{clousureofsubdry} is strict.
For example consider the injective map
                                                            \begin{gather*}
                                                              \begin{array}{rcl}
                                                              \Hp\cup\set{\infty}
                                                              &\lra&
                                                              \big(\C\cup\set{\infty}\big)\times\big(\C\cup\set{\infty}\big)
                                                              \\
                                                              z
                                                              &\longmapsto&
                                                              (z^2,z^3)
                                                              \end{array}
                                                            \end{gather*}
which has singularities on the boundary.
The complement of the set of injective points along the boundary has the following description:
                                                            \begin{gather}
                                                              \label{dwithoutinjubdr}
                                                              \partial\D\setminus\inj(u\bd)=
                                                              \crit(u\bd)\cup S_1^{\partial}(u)\supset
                                                              \overline{S_1^{\partial}(u)}.
                                                            \end{gather}
Therefore, we see again that $\inj(u\bd)$ is open.

%%%%%%%%%%%%%%%%%%%%%%%%%%%%%%%%%%%%%%%%%%%%%%%%%%%%%%%%%%%%%%%%%%%%%%
%%%%%%%%%%%%%%%%%%%%%%%%%%%%%%%%%%%%%%%%%%%%%%%%%%%%%%%%%%%%%%%%%%%%%%

\section[Different notions along the boundary]{Different
  notions of simplicity along the boundary\label{simplicityalongtheboundary}}

In this section we give a first criterion for a simple holomorphic disc
to have the half-annulus property.
By Theorem \ref{halfannulusproperty} the half-annulus property is equivalent
to be strongly simple along the boundary
which implies simplicity as we saw in Proposition \ref{frombdrtoint}.
So we will find out when a simple holomorphic disc is strongly simple along the boundary.

Recall, that
a holomorphic disc $u$ is called strongly simple along the boundary
if $\inj(u)\cap\partial\D$ is dense in $\partial\D$.
Ignoring mixed self-intersections of $u$ we find a second notion of simplicity.

\begin{defi}
  \label{defofweaksimpleonbdry}
  We call a $J$-holomorphic disc $u$ {\bf weakly simple along the boundary}
  or {\bf weakly} $\partial$-{\bf simple}
  if $\inj(u\bd)$ is dense in $\partial\D$.
\end{defi}

This is the version of simplicity
which corresponds to the notion of being somewhere injective along the boundary.
By \eqref{weakvsstrong} we have that any strongly $\partial$-simple $J$-holomorphic disc
is weakly $\partial$-simple.
The converse is not true, as the map $z\mapsto z^3$ defined on the upper half-plane shows.
This holomorphic disc is weakly $\partial$-simple but not even simple (while it is somewhere injective).
So we introduced in fact a weaker notion.

\begin{prop}
  \label{frombdrtoint2}
  Any simple and weakly $\partial$-simple holomorphic disc is strongly $\partial$-simple.
\end{prop}

\begin{proof}
  Before we come to the actual proof we remark
  that a holomorphic disc $u$ is not strongly $\partial$-simple
  if and only if $\inj(u)\cap\partial\D$ is not dense in $\partial\D$.
  This is the same as the complement $\partial\D\setminus\inj(u)$ has an interior point.
  Because of the fact that $\crit(u)$ is finite, $u$ is not strongly $\partial$-simple
  if and only if $S_1(u)\cap\partial\D$ has an interior point in $\partial\D$,
  see \eqref{dwithoutinju}.
  Hence, there exists a non-empty open subset $I$ of $\partial\D$ such that
  $u\restrict_I$ is an embedding,
  and such that $u(I)\subset u(\D\setminus I)$.
  
  Now the aim is to show that a weakly but not strongly $\partial$-simple holomorphic disc can not be simple.
  So we assume that the holomorphic disc $u$ from our preparatory remark is weakly $\partial$-simple as well.
  i.e.\ we assume in addition that $\inj(u\bd)$ is dense in $\partial\D$.
  By Proposition \ref{injuopenbound} this set is open in $\partial\D$ too,
  so that the intersection
  $\inj(u\bd)\cap I$ is an open and non-empty subset of $\partial\D$.
  But this means, we can assume that
  $u^{-1}\big(u(I)\big)\cap\partial\D=I$.
  In particular, this excludes the case $u(I)\subset u(\partial\D\setminus I)$.
  Consequently, the open subset $I$ of $\partial\D$ satisfies $u(I)\subset u\big(B_1(0)\big)$.
  It follows from Lemma \ref{okbranches} below that $u$ can not be simple.
\end{proof}

\begin{lem}
  \label{isolemib}
  Let $(M,J)$ denote an almost complex manifold of real dimension $2n$
  which contains a totally real submanifold $L$.
  Let 
  \begin{gather*}
    u:\big(B_1(0),0\big)\lra(M,p)
    \quad\text{{\it and}}\quad
    v:\big(\D^+,[-1,1],0\big)\lra(M,L,p)
  \end{gather*}
  be $J$-holomorphic embeddings.
  If $v([-1,1])\subset u\big(B_1(0)\big)$
  then there exists an open neighbourhood $V$ of zero in $\D^+$ such that
  $v(V)\subset u\big(B_1(0)\big)$.
\end{lem}

\begin{proof}
  The proof is essentially the same as for Lemma \ref{isolembb}.
  By \cite[Lemma 2.4.2]{mcdsal04} we can assume that the map $u$ satisfies
  $u(z)=(z,0,\ldots,0)$  for all $z\in B_1(0)$
  and the almost complex structure $J$ is equal to $i$ on $B_1(0)\times\set{0}$,
  where the splitting is taken w.r.t. $\C\times\C^{n-1}$.
  Writing $v=(v_1,\tilde{v})$ the projected map $\tilde{v}$
  has in general no totally real boundary condition.
  But using the condition $v([-1,1])\subset u\big(B_1(0)\big)$ as posted in the lemma
  we have $\tilde{v}([-1,1])=\set{0}$.
  By the same computation as in the proof of Lemma \ref{isolembb} we get
  \begin{gather*}
    0=
    \partial_s\tilde{v}+i\partial_t\tilde{v}+\tilde{B}\tilde{v}
  \end{gather*}
  for a suitable smooth (real) $\C^{n-1}$-endomorphism field $\tilde{B}$
  along an open neighbourhood $V$ in $\D^+$ about $0$.
  The relative Carleman similarity principle (see \cite[Theorem A.2]{abb04}) implies now
  that $\tilde{v}=0$ on $V$.
  Hence, the embedding $\varphi\define v_1$ is analytic on the interior of $V$.
  We get $v=u\circ\varphi$ on $V$,
  which proves the claim.
\end{proof}

With this observation done we can finish the proof of Proposition \ref{frombdrtoint2} with the following lemma.
The reason why we can not apply Lemma \ref{isolemib} directly is that we have to make sure
that $u(I)$ is contained completely in a sheet of $u\big(B_1(0)\big)$.

\begin{lem}
  \label{okbranches}
  Let $u$ be a $J$-holomorphic disc.
  If there exists an open non-empty subset $I$ of $\partial\D$ such that $u(I)\subset u\big(B_1(0)\big)$
  then $u$ is not simple.
\end{lem}

\begin{proof}
  By finiteness of $\crt(u)$ we can assume that $I\cap\crt(u)=\emptyset$.
  Additionally, we can assume that $u\restrict_I$ is an embedding.
  In view of Proposition \ref{pointsfinite}
                                                            \begin{gather*}
                                                              \ell(z)\define
                                                              \#\Big(u^{-1}\big(u(z)\big)\cap B_1(0)\Big)
                                                            \end{gather*}
  is for all $z\in I$ a (finite) number.
  
  We claim that we can take $z_0\in I$ in such a way that $\ell(z_0)=1$.
  Otherwise, we have that $\ell(z)\geq 2$ for all $z\in I$.
  In particular, we find sequences $z_{\nu}\ra z_0$ in $I$ and $w^1_{\nu}\ra w_1$, $w^2_{\nu}\ra w_2$
  in $\D$ such that $w^1_{\nu}\neq w^2_{\nu}$ lie in $B_1(0)$
  and $u(z_{\nu})=u(w^1_{\nu})=u(w^2_{\nu})$ for all $\nu\in\N$.
  Because of $I\cap\crt(u)=\emptyset$ the points $z_0$, $w_1$ and $w_2$ are pairwise distinct.
  Further, non of the $w_j$ is contained in $\partial\D$,
  because otherwise by Lemma \ref{nicechart} we would get $du(w_j)=0$.
  But this is not possible because of $I\cap\crt(u)=\emptyset$ and $u(z_0)=u(w_j)$.
  So both points $w_1$ and $w_2$ are in $B_1(0)$.
  Lemma \ref{weakversion} shows now that $u$ is not simple and we are done.
  So we are left with the case, where $\ell(z_0)=1$.
  
  Fix $z_0\in I$ with $\ell(z_0)=1$.
  There exist a positive real number $\varepsilon$ and
  an interior point $z_1$ in $B_1(0)$,
  such that $B_{\varepsilon}(z_1)\subset B_1(0)$,
                                                            \begin{gather}
                                                              \label{allinteriors}
                                                              u^{-1}\big(u(z_0)\big)\cap B_1(0)=\set{z_1}
                                                            \end{gather}
  and $u\restrict_{B_{\varepsilon}(z_1)}$ is an embedding.
  We will show that we can shrink the interval $I\subset\partial\D$
  such that $u(I)$ is contained in $u\big(B_{\varepsilon}(z_1)\big)$.
  Suppose there is a sequence $z^{\nu}$ in $I$ such that $z^{\nu}\ra z_0$ and
                                                            \begin{gather}
                                                              \label{notinthebranches}
                                                              u(z^{\nu})\notin
                                                              u\big(B_{\varepsilon}(z_1)\big)
                                                            \end{gather}
  for all $\nu\in\N$.
  Because of $u(I)\subset u\big(B_1(0)\big)$ we find a sequence
                                                            \begin{gather*}
                                                              w^{\nu}\in
                                                              B_1(0)\setminus
                                                              B_{\varepsilon}(z_1)
                                                              \quad\text{such that}\quad
                                                              u(z^{\nu})=u(w^{\nu})
                                                            \end{gather*}
  for all $\nu\in\N$.
  We can assume that
                                                            \begin{gather*}
                                                              w^{\nu}\lra w_0
                                                              \quad\text{in}\quad
                                                              \D\setminus
                                                              B_{\varepsilon}(z_1).
                                                            \end{gather*}
  Hence, $u(z_0)=u(w_0)$ and with \eqref{allinteriors} we have $w_0\in\partial\D$, a mixed boundary-intersection.
  Again with Lemma \ref{nicechart} this implies that $du(w_0)=0$, i.e.\ that $z_0\in\crt(u)$.
  This is a contradiction.
  So \eqref{notinthebranches} must be wrong and we obtain
                                                            \begin{gather*}
                                                              u(I)\subset
                                                              u\big(B_{\varepsilon}(z_1)\big).
                                                            \end{gather*}
  Now the claim follows from Lemma \ref{isolemib}.
\end{proof}

Together with Proposition \ref{frombdrtoint} we obtain:

\begin{cor}
  \label{frombdrtoint3}
  For a simple holomorphic disc both notions of simplicity along the boundary coincide;
  the disc is weakly $\partial$-simple
  if and only if the disc is strongly $\partial$-simple.
\end{cor}

%%%%%%%%%%%%%%%%%%%%%%%%%%%%%%%%%%%%%%%%%%%%%%%%%%%%%%%%%%%%%%%%%%%%%%
%%%%%%%%%%%%%%%%%%%%%%%%%%%%%%%%%%%%%%%%%%%%%%%%%%%%%%%%%%%%%%%%%%%%%%

\section[Self-matching discs]{Self-matching holomorphic discs\label{notselfmatching}}

In this section we give a second characterisation for a simple holomorphic disc to have the half-annulus property.
The obstruction to have it is essentially explained by the following example:
Consider the holomorphic half-discs defined by
$u(z)=z$ and $v(z)=-z$ for all $z\in\D^+$.
The images of $u$ and $v$ intersect along the real line
such that they match to get the analytic function $w(z)=z$ for $z\in\D$.

More generally we have the following gluing result:

\begin{lem}
  \label{smoothmatching}
  Let $M$ be an almost complex manifold and $L$ be a totally real submanifold of $M$.
  Consider embedded holomorphic half-discs
  \[
  u,v:\big(\D^+,[-1,1],0\big)\lra(M,L,p)
  \]
  and assume that $v([-1,1])=u([-1,1])$.
  If the sign of the boundary-intersection $\delta_{u,v}$ at $(0,0)$ is negative
  then there exist a smooth map $w:(\D,0)\lra(M,p)$,
  a diffeomorphism $\varphi$ between pointed neighbourhoods in $\C$ of zero,
  and an open neighbourhood $U$ of zero in $\C$
  such that $w(U^+)=u(U^+)$ and $w(U^-)=v\circ\varphi(U^-)$,
  where $U^{\pm}\define\D^{\pm}\cap U$.
\end{lem}

\begin{proof}
  As in the proof of Lemma \ref{isolembb}
  (using the notation from there as well)
  we can assume that $u$ is flat,
  i.e.\ $u(z)=(z,0,\ldots,0)$ in $\C\times\C^{n-1}$ for all $z\in\D^+$,
  and that $v_1$ is an embedding.
  By assumption $\delta_{u,v}$ is negative
  so that we can assume $v_1(z)\in\set{\im(z)<0}$ for all $z\in\set{\im(z)>0}\cap\D^+$.
  Moreover, our first assumption translates into $v_1([-1,1])=[-1,1]$ and $\tilde{v}([-1,1])=\set{0}$.
  Consequently,
  \[
  v\circ\varphi(z)=\big(z,\tilde{v}\circ\varphi(z)\big)
  \]
  with
  \[
  \varphi\define v_1^{-1}:v_1(\D^+)\lra\D^+
  \]
  for all $z\in\D^+$, for which $\abs{z}$ small enough.
  Because of $J=i$ on $\D^+\times\set{0}$ and the boundary condition $\tilde{v}([-1,1])=\set{0}$
  all partial derivatives of $\tilde{v}$ vanish along $[-1,1]$,
  as an induction shows.
  This implies that the map $w$ defined by
                                                            \begin{gather*}
                                                              w(z)\define
                                                              \begin{cases}
                                                                u(z)=(z,0) &
                                                                \quad\text{ for all}\quad
                                                                z\in\set{\im(z)\geq0}\cap B_{\varepsilon}(0)
                                                                \\
                                                                v\circ\varphi(z) &
                                                                \quad\text{ for all}\quad
                                                                z\in\set{\im(z)<0}\cap B_{\varepsilon}(0)
                                                              \end{cases}
                                                            \end{gather*}
  is a smooth embedding for $\varepsilon>0$ small enough.
  This proves the claim.
\end{proof}

\begin{rem}
  On the open set $U$ we can define a smooth complex structure by
                                                            \begin{gather*}
                                                              j\define
                                                              dw^{-1}_{w}\circ
                                                              J_{w}\circ
                                                              dw
                                                            \end{gather*}
  such that $j=i$ on $U^+$.
  By the theorem of Newlander-Nirenberg
  (see \cite[Theorem E.3.1]{mcdsal04})
  we find a diffeomorphism $\psi$ of $U$ such that $w\circ\psi$
  is holomorphic on $U$.
\end{rem}

In other words, if for a holomorphic disc $u$ there are two disjoint boundary segments
on which $u$ is immersed and takes the same values
there are two possibilities for $u$.
Either, the sign is negative
and (by the proceeding lemma)
the holomorphic disc $u$ is self-gluing along the boundary segments,
or the sign is positive and (by Lemma \ref{isolembb}) the holomorphic disc $u$ overlaps.
In either case the holomorphic disc $u$ can not be weakly simple along the boundary:

\begin{prop}
  \label{weakisnotselfm}
  A $J$-holomorphic disc $u$ is weakly $\partial$-simple if and only if there are no two nonempty
  disjoint open subsets $I_1$ and $I_2$ of $\partial\D$
  such that $u(I_1)=u(I_2)$.
\end{prop}

\begin{proof}
  Recall that weak $\partial$-simplicity is equivalent to the fact that $\inj(u\bd)$ is dense in $\partial\D$,
  i.e.\ the complement $\partial\D\setminus\inj(u\bd)$ has no interior point.
  Therefore, (and by finiteness of $\crit(u\bd)$)
  $u$ is not weakly simple along the boundary if and only if
  $S_1^{\partial}(u)$ has an interior point in $\partial\D$,
  see \eqref{dwithoutinjubdr},
  i.e.\ if there exists an open, non-empty, and connected subset $I$ of $\partial\D$
  such that
                                                            \begin{gather}
                                                              \label{startingpoint}
                                                              u(I)\subset
                                                              u(\partial\D\setminus I).
                                                            \end{gather} 
  Therefore it is enough to show that if \eqref{startingpoint} holds true
  we find nonempty
  disjoint open subsets $I_1$ and $I_2$ of $\partial\D$
  such that $u(I_1)=u(I_2)$.
  
  So let us assume \eqref{startingpoint}.
  In addition we can assume
  that $I\cap\crt(u)=\emptyset$
  and that $u\restrict_I$ is an embedding.
  In view of Proposition \ref{pointsfinite}
                                                            \begin{gather*}
                                                              k(z)\define
                                                              \#\Big(u^{-1}\big(u(z)\big)\cap (\partial\D\setminus I)\Big)
                                                            \end{gather*}
  is a (finite) number for any $z\in I$.
  
  If $k(z_0)\geq2$ for some $z_0\in I$,
  we find sequences $z_{\nu}\ra z_0$ in $I$ and $w^1_{\nu}\ra w_1$, $w^2_{\nu}\ra w_2$
  in $\partial\D\setminus I$ such that $w^1_{\nu}\neq w^2_{\nu}$
  and $u(z_{\nu})=u(w^1_{\nu})=u(w^2_{\nu})$ for all $\nu\in\N$.
  Because of $I\cap\crt(u)=\emptyset$ the points $z_0$, $w_1$ and $w_2$ are pairwise distinct.
  Hence, there are $3$ intersection pairs
                                                            \begin{gather*}
                                                              \bigset{
                                                              (z_0,w_1),
                                                              (z_0,w_2),
                                                              (w_1,w_2)}
                                                          \end{gather*}
  such that Lemma \ref{isolembb} applies at least to $2$ of them.
  Consequently, the overlapping regions already imply the existence of $I_1$ and $I_2$ as stated in the proposition.
  
  So we are left with the case $z_0\in I$ and $k(z_0)=1$.
  We claim that there is an open subset $S$ of $\partial\D$ such that $u\restrict_S$ is an embedding,
  $I\cap S=\emptyset$ and $u(I)=u(S)$.
  Now, there exist a positive real number $\varepsilon$ and
  a boundary point $z_1$ in $\partial\D\setminus\set{z_0}$
  such that $u(z_0)=u(z_1)$,
                                                            \begin{gather*}
                                                              I_{\varepsilon}(z_1)\define\partial\D\cap B_{\varepsilon}(z_1)
                                                            \end{gather*}
  is contained in $\partial\D\setminus I$,
  and $u\restrict_{I_{\varepsilon}(z_1)}$ is an embedding.
  We will show that we can shrink the interval $I\subset\partial\D$
  such that $u(I)$ is contained in $u\big(I_{\varepsilon}(z_1)\big)$.
  Suppose that there is a sequence $z^{\nu}$ in $I$ such that $z^{\nu}\ra z_0$ and
                                                            \begin{gather}
                                                              \label{notinthebranchesbdry}
                                                              u(z^{\nu})\notin
                                                              u\big(I_{\varepsilon}(z_1)\big)
                                                            \end{gather}
  for all $\nu\in\N$.
  Because of \eqref{startingpoint} we find a sequence
                                                            \begin{gather*}
                                                              w^{\nu}\in
                                                              \partial\D\setminus
                                                              \big(I\cup  I_{\varepsilon}(z_1)\big)
                                                              \quad\text{such that}\quad
                                                              u(z^{\nu})=u(w^{\nu})
                                                            \end{gather*}
  for all $\nu\in\N$.
  We can assume that
                                                            \begin{gather*}
                                                              w^{\nu}\lra w_0
                                                              \quad\text{in}\quad
                                                              \partial\D\setminus
                                                              \big(I\cup  I_{\varepsilon}(z_1)\big).
                                                            \end{gather*}
  Hence, $u(z_0)=u(w_0)$ but $w_0\notin\set{z_0,z_1}$
  implying that $k(z_0)\geq 2$.
  In view of \eqref{notinthebranchesbdry} this contradiction shows
                                                            \begin{gather*}
                                                              u(I)\subset
                                                              u\big(I_{\varepsilon}(z_1)\big).
                                                            \end{gather*}
  The claim follows then by setting
                                                            \begin{gather*}
                                                              S\define
                                                              \Big(u\restrict_{I_{\varepsilon}(z_1)}\Big)^{-1}
                                                              \big(u(I)\big).
                                                            \end{gather*}
  This finishes the proof.
\end{proof}

We have shown that 
any weakly $\partial$-simple holomorphic disc $u$ is not self-matching
and vice versa.
Here is the relevant definition:

\begin{defi}
  \label{def:selfm}
  We will say that a $J$-holomorphic disc $u$ is {\bf self-matching}
  if there are non-empty open connected disjoint subsets
  $I_1$ and $I_2$ of $\partial\D\setminus\crit(u\bd)$
  such that $u(I_1)=u(I_2)$.
\end{defi}

This terminology (and Proposition \ref{frombdrtoint2}) finally gives us a criterion for a holomorphic disc to
have the half-annulus property,
i.e.\ to be simple along the boundary.
It says that in order to verify the half-annulus property
it is the same to verify simplicity (in the sense of Proposition \ref{equivcharofsimple})
for interior and boundary points separately.

\begin{cor}
  \label{onself-matching3}
  A $J$-holomorphic disc is strongly $\partial$-simple if and only if
  it is simple and not self-matching.
\end{cor}

\begin{proof}
  By Proposition \ref{frombdrtoint} any strongly $\partial$-simple holomorphic disc is simple.
  This shows the first half of the 'only if' part.
  For the rest recall that by Corollary \ref{frombdrtoint3}
  for a simple holomorphic disc being weakly or strongly simple along the boundary is the same.
  But by Proposition \ref{weakisnotselfm} weak $\partial$-simplicity is characterised by
  the not self-matching property.
\end{proof}

\begin{proof}[{\bf Proof of Theorem \ref{halfannulusintro}}]
  It is claimed that a holomorphic disc $u$
  has the half-annulus property
  (which is equivalent to the strong simplicity along the boundary by Theorem \ref{halfannulusproperty})
  if and only if $u$ is simple and not self-matching.
  Therefore, the theorem follows from Corollary \ref{onself-matching3} above.
\end{proof}

{\bf Acknowledgment.}
The research presented in this article
is part of my doctoral thesis \cite{zehm08}.
I am grateful to my supervisor Matthias Schwarz for his steady help and encouragement.
I thank Hansj\"org Geiges for comments on the first version of this article.

%%%%%%%%%%%%%%%%%%%%%%%%%%%%%%%%%%%%%%%%%%%%%%%%%%%%%%%%%%%%%%%%%%%%%%
%%%%%%%%%%%%%%%%%%%%%%%%%%%%%%%%%%%%%%%%%%%%%%%%%%%%%%%%%%%%%%%%%%%%%%

\end{document}